\documentclass[11pt]{amsart}

\usepackage{amsmath, amssymb}
\usepackage{amsfonts}
\usepackage{graphicx}
\usepackage{mathrsfs}
\usepackage[arrow,matrix,curve,cmtip,ps]{xy}
\usepackage{amsthm}
\usepackage{enumerate}
\usepackage{color}
\usepackage[colorlinks]{hyperref}
\usepackage{tikz}

\allowdisplaybreaks

\newtheorem{thm}{Theorem}[section]
\newtheorem{lem}[thm]{Lemma}
\newtheorem{prop}[thm]{Proposition}
\newtheorem{cor}[thm]{Corollary}
\newtheorem*{theorem*}{Theorem}
\theoremstyle{remark}
\newtheorem{rem}[thm]{Remark}

\newtheorem{defn}[thm]{Definition}

\numberwithin{equation}{section}

\newcommand{\Z}{\mathbb{Z}}
\newcommand{\N}{\mathbb{N}}

\newcommand{\C}{\mathbb{C}}
\newcommand{\g}{\mathcal{G}_{G,\Lambda}}
\newcommand{\La}{\Lambda}
\newcommand{\la}{\lambda}

\newcommand{\ep}{\mathrm{EP}_R(G,\Lambda)}
\newcommand{\e}{\mathcal{E}}
\newcommand{\f}{\mathbb{F}_R(X)}
\newcommand{\lag}{\Lambda \bowtie G}
\newcommand{\q}{\mathcal{Q}_R^{\mathrm{alg}}(\Lambda\bowtie G)}
\newcommand{\mj}{\mathcal{J}}

\begin{document}
\title[self-similar $k$-graph algebras]{Exel-Pardo algebras of self-similar $k$-graphs}

\author[Hossein Larki]{Hossein Larki}

\address{Department of Mathematics\\
Faculty of Mathematical Sciences and Computer\\
Shahid Chamran University of Ahvaz\\
P.O. Box: 83151-61357\\
Ahvaz\\
 Iran}
\email{h.larki@scu.ac.ir}


\date{\today }


\begin{abstract}
We introduce the Exel-Pardo $*$-algebra $\mathrm{EP}_R(G,\Lambda)$ associated to a self-similar $k$-graph $(G,\Lambda,\varphi)$. We also prove the $\mathbb{Z}^k$-graded and Cuntz-Krieger uniqueness theorems for such algebras and investigate their ideal structure. In particular, we modify the graded uniqueness theorem for self-similar 1-graphs, and then apply it to present $\mathrm{EP}_R(G,\Lambda)$ as a Steinberg algebra and to study the ideal structure.
\end{abstract}

\maketitle


\section{introduction}

To give a unified framework like graph $C^*$-algebras for the Katsura's \cite{kat08} and Nekrashevyche's algebras \cite{nek04,nek05}, Exel and Pardo introduced self-similar graphs and their $C^*$-algebras in \cite{exe17}. They then associated an inverse semigroup and groupoid model to this class of $C^*$-algebras and studied structural features by underlying self-similar graphs. Note that although only finite graphs are considered in \cite{exe17}, many of arguments and results may be easily generalized for countable row-finite graphs with no sources (see \cite{exe18,lar21} for example). Inspired from \cite{exe17}, Li and Yang in \cite{li18, li19} introduced self-similar action of a discrete countable group $G$ on a row-finite $k$-graph $\La$. They then associated a universal $C^*$-algebra $\mathcal{O}_{G,\La}$ to $(G,\La)$ satisfying specific relations.

The algebraic analogues of Exel-Pardo $C^*$-algebras, denoted by $\mathcal{O}_{(G,E)}^{\mathrm{alg}}$ in \cite{cla18} and by $L_R(G,E)$ in \cite{haz19}, were introduced and studied in \cite{cla18,haz19}. In particular, Hazrat et al. proved a $\Z$-graded uniqueness theorem, and gave a model of Steinberg algebras for $L_R(G,E)$ \cite{haz19}. The initial aims to write the present paper are to give a much easier proof for \cite[Theorem B]{haz19} (a groupoid model for $L_R(G,E)$) using the $\mathbb{Z}$-graded uniqueness theorem, and then to study the ideal structure. However, we do these here, among others, for a more general class of algebras associated to self-similar higher rank graphs $(G,\La)$, which is introduced in Section 2.

This article is organized as follows. Let $R$ be a unital commutative $*$-ring. In Section 2, we introduce a universal $*$-algebra $\ep$ of a self-similar $k$-graph $(G,\La)$ satisfying specific properties, which is called the \emph{Exel-Pardo algebra of $(G,\La)$}. Our algebras are the higher rank generalization of those in \cite[Theorem 1.6]{haz19} and the algebraic analogue of  $\mathcal{O}_{G,\La}$ \cite{li18,li19}. Moreover, this class of algebras includes many important known algebras such as the algebraic Katsura algebras \cite{haz19}, Kumjian-Pask  algebras \cite{pin13}, and the quotient boundary algebras $\q$ of a Zappa-Sz$\acute{\mathrm{e}}$p product $\lag$ introduced in Section 3.  In section 3, we give a specific example of Exel-Pardo algebras using boundary quotient algebras of semigroups. Indeed, for a single-vertex self-similar $k$-graph $(G,\La)$, the Zappa-Sz\'{e}p product $\lag$ is a cancellative semigroup. We prove that the quotient boundary algebra $\mathcal{Q}_R^{\mathrm{alg}}(\lag)$ (defined in Definition \ref{defn3.1}) is isomorphic to $\ep$. Section 4 is devoted to proving a graded uniqueness theorem for the Exel-Pardo algebras. Note that using the description in Proposition \ref{prop2.7} below, there is a natural $\Z^k$-grading on $\ep$. Then, in Theorem \ref{grad-uniq-thm}, a $\Z^k$-graded uniqueness theorem is proved for $\ep$ which generalizes and modifies \cite[Theorem A]{haz19}. In particular, we will see in Sections 5 and 6 that this modification makes it more applicable.

In Sections 5 and 6, we assume that our self-similar $k$-graphs are pseudo free (Definition \ref{defn5.1}). In Section 5, we prove that every Exel-Pardo algebra $\ep$ is isomorphic to the Steinberg algebra $A_R(\g)$, where $\g$ is the groupoid introduced in \cite{li18}. We should note that the proof of this result is completely different from that of \cite[Theorem B]{haz19}. Indeed, the main difference between the proof of Theorem \ref{steib-algeb-thm} and that of \cite[Theorem B]{haz19} is due to showing the injectivity of defined correspondence. In fact, in \cite[Theorem B]{haz19}, the authors try to define a representation for $\mathcal{S}_{(G,E)}$ in $\mathrm{EP}_R(G,E)$ while we apply our graded uniqueness theorem, Theorem \ref{grad-uniq-thm}. This gives us an easier proof for Theorem \ref{steib-algeb-thm}, even in the 1-graph case. Finally, in Section 6, we investigate the ideal structure of $\ep$. Using the Steinberg algebras, we can define a conditional expectation $\e$ on $\ep$ and then characterize basic, $\Z^k$-graded, diagonal-invariant ideals of $\ep$ by $G$-saturated $G$-hereditary subsets of $\La^0$. These ideals are exactly basic, $Q(\N^k,G)\rtimes_{\mathcal{T}}\Z^k$-graded ideals of $\ep$.

\subsection{Notation and terminology}
Let $\mathbb{N}=\{0,1,2,\ldots\}$. For $k\geq 1$, we regard $\mathbb{N}^k$ as an additive semigroup with the generators $e_1,\ldots,e_k$. We use $\leq $ for the partial order on $\mathbb{N}^k$ given by $m\leq n$ if and only if $m_i\leq n_i$ for $1\leq i\leq k$. We also write $m\vee n$ and $m\wedge n$ for the coordinate-wise maximum and minimum, respectively.

A \emph{$k$-graph} is a countable small category $\La=(\La^0,\La,r,s)$ equipped with a \emph{degree functor} $d:\La \rightarrow \N^k$ satisfying the unique factorization property: for $\mu\in \La$ and $m,n\in \N^k$ with $d(\mu)=m+n$, then there exist unique $\alpha,\beta\in \La$ such that $d(\alpha)=m$, $d(\beta)=n$, and $\mu=\alpha \beta$. We usually denote $\mu(0,m):=\alpha$ and $\mu(m,d(\mu)):=\beta$. We refer to $\La^0$ as the vertex set and define $\La^n:=\{\mu\in\La: d(\mu)=n\}$ for every $n\in \N^k$. For $A,B\subseteq \La$, define $AB=\{\mu\nu:\mu\in A,\nu\in B, ~\mathrm{and}~ s(\mu)=r(\nu)\}$. Also, for $\mu,\nu\in \La$, define $\La^{min}(\mu,\nu)=\{(\alpha,\beta)\in \La\times \La: \mu\alpha=\nu\beta, d(\mu\alpha)=d(\mu)\vee d(\nu)\}$.

We say that $\La$ is \emph{row-finite} if $v\La^n$ is finite for all $n\in \N^k$ and $v\in \La^0$. A \emph{source} in $\La$ is a vertex $v\in \La^0$ such that $v\La^{e_i}=\emptyset$ for some $1\leq i\leq k$.

{\bf Standing assumption.} Throughout the article, we work only with row-finite $k$-graphs without sources.

Let $\Omega_k:=\{(m,n)\in \N^k\times \N^k:m\leq n\}$. By defining $(m,n).(n,l):=(m,l)$, $r(m,n):=(m,m)$, and $s(m,n):=(n,n)$, then $\Omega_k$ is a row-finite $k$-graph without sources. A graph homomorphism $x:\Omega_k\rightarrow \La$ is called an \emph{infinite path of $\La$} with the range $r(x)=x(0,0)$, and we write $\La^\infty$ for the set of all infinite paths of $\La$.


\section{Exel-Pardo algebras of self-similar $k$-graphs}

In this section, we associate a $*$-algebra to a self-similar $k$-graph as the algebraic analogue of \cite[Definition 3.9]{li19}. Let us first review some definitions and notations.

Following \cite{haz19}, we consider $*$-algebras over $*$-rings. Let $R$ be a unital commutative $*$-ring. Recall that a \emph{$*$-algebra over $R$} is an algebra $A$ equipped with an involution such that $(a^*)^*=a$, $(ab)^*=b^* a^*$, and $(ra+b)^*=r^* a^*+b^*$ for all $a,b\in A$ and $r\in R$. Then $p\in A$ is called a \emph{projection} if $p^2=p=p^*$, and $s\in A$ a \emph{partial isometry} if $s=ss^*s$.

\begin{defn}[{\cite[Definition 3.1]{pin13}}]
Let $\La$ be a row-finite $k$-graph without sources. A \emph{Kumjian-Pask $\La$-family} is a collection $\{s_{\mu}:\mu\in \La\}$ of partial isometries in a $*$-algebra $A$ such that
\begin{enumerate}[(KP1)]
  \item $\{s_v:v\in \La^0\}$ is a family of pairwise orthogonal projections;
  \item $s_{\mu\nu}=s_\mu s_\nu$ for all $\mu,\nu\in\La$ with $s(\mu)=r(\nu)$;
  \item $s_\mu^* s_\mu=s_{s(\mu)}$ for all $\mu\in \La$; and
  \item $s_v=\sum_{\mu\in v\La^n}s_\mu s_\mu^*$ for all $v\in \La^0$ and $n\in \N^k$.
\end{enumerate}
\end{defn}

\subsection{Self-similar $k$-graphs and their algebras}
Let $\La$ be a row-finite $k$-graph without sources. An \emph{automorphism of $\La$} is a bijection $\psi:\La\rightarrow \La$ such that $\psi(\La^n)\subseteq \La^n$ for all $n\in \N^k$ with the properties $s\circ \psi=\psi \circ s$ and $r\circ \psi=\psi \circ r$. We denote by $\mathrm{Aut}(\La)$ the group of automorphisms on $\La$. Furthermore, if $G$ is a countable discrete group, an \emph{action} of $G$ on $\La$ is a group homomorphism $g\mapsto \psi_g$ from $G$ into $\mathrm{Aut}(\La)$.

\begin{defn}[{\cite{li18}}]\label{defn2.2}
Let $\La$ be a row-finite $k$-graph without sources and $G$ a discrete group with identity $e_G$. We say that a triple $(G,\La,\varphi)$ is a \emph{self-similar $k$-graph} whenever the following properties hold:
\begin{enumerate}
  \item $G$ acts on $\La$ by a group homomorphism $g\mapsto \psi_g$. We prefer to write $g\cdot \mu$ for $\psi_g(\mu)$ to ease the notation.
  \item $\varphi:G\times \La\rightarrow G$ is a 1-cocycle for the action $G\curvearrowright \La$ such that for every $g\in G$, $\mu,\nu\in \La$ and $v\in \La^0$ we have
  \begin{itemize}
    \item[(a)] $\varphi(gh,\mu)=\varphi(g,h \cdot \mu)\varphi(h,\mu)  \hspace{5mm}$ (the 1-cocycle property),
    \item[(b)] $g\cdot (\mu\nu)=(g\cdot \mu)(\varphi(g,\mu) \cdot \nu) \hspace{5mm}$ (the self-similar equation),
    \item[(c)] $\varphi(g,\mu\nu)=\varphi(\varphi(g,\mu),\nu)$, and
    \item[(d)] $\varphi(g,v)=g$.
  \end{itemize}
\end{enumerate}
For convenience, we usually write $(G,\La)$ instead of $(G,\La,\varphi)$.
\end{defn}

\begin{rem}
In \cite{li18}, the authors used the notation $g|_\mu$ for $\varphi(g,\mu)$. However, we would prefer to follow \cite{exe17,exe18,haz19} for writing $\varphi(g,\mu)$.
\end{rem}

\begin{rem}
If in equation (2)(a) of Definition \ref{defn2.2}, we set $g=h=e_G$, then we get $\varphi(e_G,\mu)=e_G$ for every $\mu\in \La$. Moreover, \cite[Lemma 3.5(ii)]{li18} shows that $\varphi(g,\mu)\cdot v=g\cdot v$ for all $g\in G$, $v\in \La^0$, and $\mu\in \La$.
\end{rem}

Now we generalize the definition of Exel-Pardo $*$-algebras \cite{haz19} to the $k$-graph case.

\begin{defn}\label{defn2.4}
Let $(G,\La)$ be a self-similar $k$-graph. An {\it Exel-Pardo $(G,\La)$-family} (or briefly {\it $(G,\La)$-family}) is a set
$$\{s_\mu:\mu\in \La\}\cup\{u_{v,g}:v\in\La^0,g\in G\}$$
in a $*$-algebra satisfying
\begin{enumerate}
  \item $\{s_\mu:\mu\in\La\}$ is a Kumjian-Pask $\La$-family,
  \item $u_{v,e_G}=s_v$ for all $v\in\La^0$,
  \item $u_{v,g}^*=u_{g^{-1} \cdot v ,g^{-1}}$ for all $v\in\La^0$ and $g\in G$,
  \item $u_{v,g}s_\mu=\delta_{v,g\cdot r(\mu)}s_{g\cdot \mu}u_{g\cdot s(\mu),\varphi(g,\mu)}$ for all $v\in\La^0$, $\mu\in\La$, and $g\in G$,
  \item $u_{v,g}u_{w,h}=\delta_{v,g\cdot w}u_{v,gh}$ for all $v,w\in\La^0$ and $g,h\in G$.
\end{enumerate}
Then the {\it Exel-Pardo algebra} $\ep$ is the universal $*$-algebra over $R$ generated by a $(G,\La)$-family $\{s_\mu,u_{v,g}\}$.
\end{defn}

Recall that the universality of $\ep$ means that for every $(G,\La)$-family $\{S_\mu,U_{v,g}\}$ in a $*$-algebra $A$, there exists a $*$-homomorphism $\phi:\ep \rightarrow A$ such that $\phi(s_\mu)=S_\mu$ and $\phi(u_{v,g})=U_{v,g}$ for all $v\in \La^0$, $\mu\in \La$, and $g\in G$. (See Subsection \ref{ss2.2} for the construction of $\ep$.)  Throughout the paper, we will denote by $\{s_\mu,u_{v,g}\}$ the $(G,\La)$-family generating $\ep$.

\subsection{The construction of $\ep$}\label{ss2.2}

Let $(G,\La)$ be a self-similar $k$-graph as in Definition \ref{defn2.2}. The following is a standard construction of a universal algebra $\ep$ subject to desired relations. Consider the set of formal symbols
$$S=\left\{S_\mu,S_\mu^*,U_{v,g},U_{v,g}^*:\mu\in \La, v\in \La^0, g\in G\right\}.$$
Let $X=w(S)$ be the collection of finite words in $S$. We equip the free $R$-module
$$\mathbb{F}_R(X):=\big\{\sum_{i=1}^l r_ix_i: l\geq 1, r_i\in R, x_i\in X\big\}$$
with the multiplication
$$\big(\sum_{i=1}^l r_i x_i\big)\big(\sum_{j=1}^{l'} s_j y_j\big):=\sum_{i,j} r_is_j x_iy_j,$$
and the involution
$$\big(\sum r_i x_i\big)^*:=\big(\sum r_i^* x_i^*\big)$$
where $x^*=s_l^*\ldots s_1^*$ for each $x=s_1\ldots s_l$. Then $\f$ is a $*$-algebra over $R$. If $I$ is the (two-sided and self-adjoint) ideal of $\f$ containing the roots of relations (1)-(5) in Definition \ref{defn2.4}, then the quotient $\f/I$ is the Exel-Pardo algebra $\ep$ with the desired universal property. Let us define $s_\mu:=S_\mu+I$ and $u_{v,g}:=U_{v,g}+I$ for every $\mu\in\La$, $v\in\La^0$, and $g\in G$. In case $(G,\La)$ is pseudo free (Definition \ref{defn5.1}), Theorem \ref{grad-uniq-thm} insures that all generators $\{s_\mu,u_{v,g}\}$ of $\ep$ are nonzero.

Proposition \ref{prop2.7} below describes the elements of $\ep$. First, see a simple lemma.

\begin{lem}\label{lem2.5}
Let $(G,\La)$ be a self-similar graph (as in Definition \ref{defn2.2}) and $\{S,U\}$ a $(G,\La)$-family. If $S_\mu U_{v,g} S_\nu^*\neq 0$ where $\mu,\nu\in \La$, $v\in \La^0$ and $g\in G$, then $s(\mu)=v=g\cdot s(\nu)$.
\end{lem}

\begin{proof}
If $a=S_\mu U_{v,g} S_\nu^*$ is nonzero, then by Definition \ref{defn2.4} we can write
\begin{align*}
S_\mu U_{v,g}S_\nu^* &= S_\mu (S_{s(\mu)}U_{v,g} ) S_\nu^*\\
&= S_\mu(U_{s(\mu),e_G} U_{v,g}) S_\nu^*\\
&=S_\mu (\delta_{s(\mu),e_G\cdot v} U_{s(\mu),g}) S_\nu^*.
\end{align*}
Now, the hypothesis $a\neq 0$ forces $s(\mu)=v$. On the other hand, a similar computation gives
\begin{align*}
a&= S_\mu (U_{v,g} S_{s(\nu)}) S_\nu^*\\
&= S_\mu( U_{v,g} U_{s(\nu),e_G}) S_\nu^*\\
&=S_\mu (\delta_{v, g\cdot s(\nu)} U_{v,g}) S_\nu^*,
\end{align*}
and thus $v=g\cdot s(\nu)$.
\end{proof}

\begin{prop}\label{prop2.7}
Let $(G,\La)$ be a self-similar graph. Then
\begin{equation}\label{eq2.1}
\ep=\mathrm{span}_R\{s_\mu u_{s(\mu),g} s_\nu^*: g\in G,~ \mu,\nu\in \La,~ \mathrm{and}~ s(\mu)=g\cdot s(\nu)\}.
\end{equation}
\end{prop}

\begin{proof}
Define $M:=\mathrm{span}_R\{s_\mu u_{s(\mu),g} s_\nu^*: g\in G,~ \mu,\nu\in \La\}$. For every $g,h\in G$ and $\mu,\nu,\alpha,\beta\in \La$ with $\alpha=\nu\alpha'$ for some $\alpha'\in \La$, the relations of Definition \ref{defn2.4} imply that
\begin{align*}
\big(s_\mu u_{s(\mu),g} s_\nu^* \big) \big(s_\alpha u_{s(\alpha),h} s_\beta^* \big) &= s_\mu u_{s(\mu),g}(s_\nu^* s_\alpha) u_{s(\alpha),h} s_\beta^*\\
&= s_\mu u_{s(\mu),g}(s_{\alpha'}) u_{s(\alpha),h} s_\beta^* \\
&= s_\mu \big( \delta_{s(\mu),g\cdot r(\alpha')} s_{g\cdot \alpha'} u_{g\cdot s(\alpha'), \varphi(g,\alpha')} u_{s(\alpha),h}\big) s_\beta^* \\
&= \delta_{s(\mu),g\cdot s(\nu)} s_{\mu(g\cdot \alpha')} \big(\delta_{g\cdot s(\alpha'),\varphi(g,\alpha') \cdot s(\alpha)} u_{g\cdot s(\alpha'), \varphi(g,\alpha') h} \big) s_\beta^*  \hspace{5mm} (\mathrm{as}~~ r(\alpha')=s(\nu)).
\end{align*}

In the case $\nu=\alpha \nu'$ for some $\nu'\in \La$, the above multiplication may be computed similarly, and otherwise is zero. Hence, $M$ is closed under multiplication. Also, we have $$\big(s_\mu u_{s(\mu),g}s_\nu^* \big)^*= s_\nu u_{g^{-1} \cdot s(\mu), g^{-1}} s_\mu^*,$$
so $M^*\subseteq M$. Since
$$s_\mu=s_\mu u_{s(\mu),e_G} s_{s(\mu)}^* \hspace{5mm} \mathrm{and} \hspace{5mm} u_{v,g}=s_v u_{v,g} s_{g\cdot v}^*$$
for all $g\in G$, $v\in \La^0$, and $\mu\in \La$, it follows that $M$ is a $*$-subalgebra of $\ep$ containing the generators of $\ep$. In light of Lemma \ref{lem2.5}, this concludes the identification (\ref{eq2.1}).
\end{proof}

\subsection{The unital case}

In case $\La$ is a $k$-graph with finite $\La^0$, we may give a better description for Definition \ref{defn2.4}. Note that this case covers all unital Exel-Pardo algebras $\ep$.

\begin{lem}
Let $(G,\La)$ be a self-similar $k$-graph and let $s_v$ be nonzero in $\ep$ for every $v\in \La^0$. Then $\ep$ is a unital algebra if and only if the vertex set $\La^0$ is finite.
\end{lem}

\begin{proof}
If $\La^0=\{v_1,\ldots, v_l\}$ is finite, then using identification (\ref{eq2.1}), $P=\sum_{i=1}^l s_{v_i}$ is the unit of $\ep$. Conversely, if $\La^0$ is infinite, then the set $\{s_v:v\in \La^0\}\subseteq \ep$ contains infinitely many mutually orthogonal projections. Now again by (\ref{eq2.1}), there is no element of $\ep$ which acts as an identity on each element of $\{s_v:v\in \La^0\}$.
\end{proof}

Note that if $\{s,u\}$ is a $(G,\La)$-family in a $*$-algebra $A$, then for each $g\in G$  we may define $u_g:=\sum_{v\in \La^0}u_{v,g}$ as an element of the multiplier algebra $\mathcal{M}(A)$ with the property $s_v u_g=u_{v,g}$ for all $v\in \La^0$ (relations (2) and (5) of Definition \ref{defn2.4} yield $s_vu_{w,g}=\delta_{v,w} u_{v,g}$). (See \cite{ara00} for the definition of multiplier algebras.) Thus relations (3) and (5) of Definition \ref{defn2.4} imply that $u:G\rightarrow \mathcal{M}(A)$, defined by $g \mapsto u_g$, is a unitary $*$-representation of $G$ on $\mathcal{M}(A)$. In particular, in case $\La^0$ is finite, $u_g$'s lie all in $A$, and we may describe Definition \ref{defn2.4} as the following:

\begin{prop}\label{prop2.9}
Let $(G,\La)$ be a self-similar $k$-graph. Suppose also that $\La^0$ is finite. Then $\ep$ is the universal $*$-algebra generated by families $\{s_\mu:\mu\in \La\}$ of partial isometries and $\{u_g:g\in G\}$ of unitaries satisfying
\begin{enumerate}
  \item $\{s_\mu:\mu\in \La\}$ is a Kumjian-Pask $\La$-family;
  \item $u:G\rightarrow \ep$, by $g\mapsto u_g$, is a unitary $*$-representation of $G$ on $\ep$, in the sense that
  \begin{itemize}
    \item[(a)] $u_g u_h=u_{gh}$ for all $g,h\in G$, and
    \item[(b)] $u^*_g= u^{-1}_g= u_{g^{-1}}$ for all $g\in G$;
  \end{itemize}
  \item $u_g s_\mu=s_{g\cdot \mu} u_{\varphi(g,\mu)}$ for all $g\in G$ and $\mu\in \La$.
\end{enumerate}
\end{prop}

\section{An example: The Zappa-Sz\'{e}p product $\lag$ and its $*$-algebra}

Let $(G,\La)$ be a self-similar $k$-graph such that $|\La^0|=1$. The $C^*$-algebra and quotient boundary $C^*$-algebra associated to the Zappa-Sz\'ep product $\lag$ as a semigroup were studied in \cite{bro14,li19-boundary}. In this section, we first define $\mathcal{Q}_R^{\mathrm{alg}}(S)$ as the algebraic analogue of the quotient boundary $C^*$-algebra $\mathcal{Q}(S)$ of a cancellative semigroup $S$. Then we show that $\q$ is isomorphic to the Exel-Pardo algebra $\ep$.

Let us recall some terminology from \cite{li12,bro14}. Let $S$ be a left-cancelative semigroup with an identity. Given $X\subseteq S$ and $s\in S$, define $sX:=\{sx:x\in X\}$ and $s^{-1}X:=\{r\in S:sr\in X\}$. Also, the set of \emph{constructible right ideals in $S$} is defined as
$$\mathcal{J}(S):=\{s^{-1}_1r_1\ldots s_l^{-1} r_l S: l\geq 1, ~ s_i,r_i\in S\}\cup \{\emptyset\}.$$
Then, a \emph{foundation set in $\mathcal{J}(S)$} is a finite subset $F\subseteq \mathcal{J}(S)$ such that for each $Y\in \mathcal{J}(S)$, there exists $X\in F$ with $X\cap Y\neq \emptyset$.

The following is the algebraic analogue of \cite[Definition 2.2]{li12}.

\begin{defn}\label{defn3.1}
Let $S$ be a left-cancelative semigroup and $R$ be a unital commutative $*$-ring. The \emph{boundary quotient $*$-algebra of $S$} is the universal unital $*$-algebra $\mathcal{Q}_R^{\mathrm{alg}}(S)$ over $R$ generated by a set of isometries $\{t_s:s\in S\}$ and a set of projections $\{q_X:X\in \mathcal{J}(S)\}$ satisfying
\begin{enumerate}
  \item $t_s t_r=t_{sr}$,
  \item $t_s q_X t_s^*=q_{sX}$,
  \item $q_S=1$ and $q_\emptyset=0$,
  \item $q_X q_Y=q_{X\cap Y}$, and moreover
  \item $\prod_{X\in F} (1-q_X)=0$
\end{enumerate}
for all $s,r\in S$, $X,Y\in \mathcal{J}(S)$, and foundation sets $F\subseteq \mathcal{J}(S)$.
\end{defn}

Let $\La$ be a $k$-graph such that $\La^0=\{v\}$. Then $\mu\nu$ is composable for all $\mu,\nu\in \La$, and hence $\La$ may be considered as a semigroup with the identity $v$. Also, the unique factorization property implies that $\La$ is cancelative.

\begin{defn}[{\cite[Definition 3.1]{bro14}}]
Let $(G,\La)$ is a single-vertex self-similar $k$-graph. If we consider $\La$ as a semigroup, then the \emph{Zappa-Sz\'ep product $\La\bowtie G$} is the semigroup $\La \times G$ with the multiplication
$$(\mu,g)(\nu,h):=(\mu(g\cdot \nu), \varphi(g,\nu) \cdot h) \hspace{5mm} (\mu,\nu\in \La ~ \mathrm{and}~ g,h\in G).$$
\end{defn}

\begin{rem}
If $\La$ is a single-vertex $k$-graph, then \cite[Lemma 3.2 (iv)]{li19-boundary} follows that
$$\mathcal{J}(\La)=\big\{\bigcup_{i=1}^l\mu_i \La : l\geq 1, \mu_i\in \La, d(\mu_1)=\cdots =d(\mu_l)\big\}.$$
\end{rem}

In order to prove Theorem \ref{thm3.6}, the following lemmas are useful.

\begin{lem}\label{lem3.4}
Let $(G,\La)$ be a self-similar $k$-graph with $\La^0=\{v\}$. Suppose that for each $\mu\in \La$, the map $g\mapsto \varphi(g,\mu)$ is surjective. Then
\begin{enumerate}
  \item $\mj(\lag)=\mj(\La)\times \{G\}$, where $\emptyset \times G:=\emptyset$.
  \item A finite subset $F\subseteq \mj(\La)$ is a foundation set if and only if $F'=F\times \{G\}$ is a foundation set in $\mj(\lag)$.
\end{enumerate}
\end{lem}

\begin{proof}
Statement (1) is just \cite[Lemma 2.13]{li19-boundary}. For (2), suppose that $F\subseteq \mj(\La)$ is a foundation set, and let $Y\times G \in \mj(\lag)$. Then there exists $X\in F$ such that $X\cap Y\neq \emptyset$. Thus $(X\times G)\cap (Y\times G)\neq \emptyset$, from which we conclude that $F\times \{G\}$ is a foundation set in $\mj(\lag)$. The converse may be shown analogously.
\end{proof}

In the following, for $\mu\in\La$ and $E\subseteq \La$ we define
$$\mathrm{Ext}(\mu;E):=\{\alpha: (\alpha,\beta)\in \La^{\mathrm{min}}(\mu,\nu) \mathrm{~ for ~ some ~} \nu\in E\}.$$

\begin{lem}\label{lem3.5}
Let $(G,\La)$ be a self-similar $k$-graph with $\La^0=\{v\}$. For every $X=\cup_{i=1}^l\mu_i \La$ and $Y=\cup_{j=1}^{l'} \nu_j \La$ in $\mj(\La)$, we have
$$X\cap Y=\cup\{\mu_i\alpha\La: 1\leq i\leq l, \alpha\in \mathrm{Ext}(\mu_i;\{\nu_i,\ldots, \nu_{l'}\})\}.$$
\end{lem}

\begin{proof}
For any $\la\in X\cap Y$, there are $\alpha,\beta\in \La$, $1\leq i\leq l$, and $1\leq j \leq l'$ such that $\la=\mu_i\alpha=\nu_j\beta$. Define
$$\alpha':=\alpha (0,d(\mu_i)\vee d(\nu_j)-d(\mu_i)) ~~~ \mathrm{and} ~~~ \beta':=\beta (0,d(\mu_i)\vee d(\nu_j)-d(\nu_j)).$$
Then the factorization property implies that $\la=\mu_i\alpha'\la'=\nu_j\beta'\la'$ where $d(\mu_i\alpha')=d(\nu_j \beta')=d(\mu_i)\vee d(\nu_j)$ and $\la'=\la(d(\mu_i)\vee d(\nu_j),d(\la))$. It follows that $\la\in \mu_i\alpha'\La$ with $\alpha'\in \mathrm{Ext}(\mu_i;\{\nu_j\})$ as desired. The reverse containment is trivial.
\end{proof}

The following result is inspired by \cite[Theorem 3.3]{li19-boundary}.

\begin{thm}\label{thm3.6}
Let $(G,\La)$ be a self-similar $k$-graph with $\La^0=\{v\}$ and let $\{s_\mu,u_g\}$ be the $(G,\La)$-family generating $\ep$ as in Proposition \ref{prop2.9}. Suppose that for every $\mu\in \La$ the map $g\mapsto \varphi(g,\mu)$ is surjective. If the family $\{t_{(\mu,g)}, q_X: (\mu,g)\in \lag, X\in \mj(\lag)\}$ generates $\q$, then there exists an $R$-algebra $*$-isomorphism $\pi:\ep \rightarrow \q$ such that $\pi(s_\mu)=t_{(\mu,e_G)}$ and $\pi(u_g)=q_{(v,g)}$ for all $\mu\in \La$ and $g\in G$.
\end{thm}

\begin{proof}
For every $\mu\in \La$ and $g\in G$, define
$$S_\mu:=t_{(\mu,e_G)} ~~ \mathrm{and} ~~ U_g:=t_{(v,g)}.$$
We will show that $\{S, U\}$ is a $(G,\La)$-family in $\q$, which is described in Proposition \ref{prop2.9}. First, for each $g\in G$ we have
$$U_g U_g^*=t_{(v,g)} ~ q_{\lag} ~ t_{(v,g)}^* = q_{(v,g)\lag}= q_{\lag}=1_{\q}.$$
So, $g\mapsto U_g$ is a unitary $*$-representation of $G$ into $\q$. Moreover, (KP1)-(KP3) can be easily checked, so we verify (KP4) for $\{s_\mu:\mu\in \La\}$. Fix some $n\in \N^k$. Then $\{\mu \La:\mu\in \La^n\}$ is a foundation set in $\mj(\La)$, and thus so is $F=\{\mu \La \times G:\mu\in\La^n\}$ in $\mj(\lag)$ by Lemma \ref{lem3.4}. Hence we have
\begin{align*}
1-\sum_{\mu\in\La^n} S_\mu S_\mu^*&=\prod_{\mu\in\La^n}(1-S_\mu S_\mu^*)  \hspace{16mm} (\mathrm{because ~} S_\mu S_\mu^*\mathrm{s ~ are ~ pairwise ~ orthogonal})\\
&=\prod_{\mu\in\La^n}(1-t_{(\mu,e_G)} ~ q_{\lag} ~ t_{(\mu,e_G)}^*)\\
&=\prod_{\mu\in\La^n}(1-q_{\mu\La \times G}) \hspace{10mm} (\mathrm{by ~ eq.~  (2) ~ of ~ Definition ~ \ref{defn3.1}}) \\
&=\prod_{X\in F} (1-q_X)=0 \hspace{10mm} (\mathrm{by ~ eq.~  (5) ~ of ~ Definition ~ \ref{defn3.1}}).
\end{align*}
Because $S_v=U_{e_G}=1$, (KP4) is verified, and therefore $\{S_\mu:\mu\in \La\}$ is a Kumjian-Pask $\La$-family. Since for each $\mu\in\La$ and $g\in G$,
$$U_g S_\mu=t_{(v,g)}t_{(\mu,e_G)} = t_{(v,g)(\mu,e_G)}=t_{(g\cdot \mu,\varphi(g,\mu))}= S_{g\cdot \mu} U_{\varphi(g,\mu)},$$
and so we have shown that $\{S,U\}$ is a $(G,\La)$-family in $\q$. Now the universality implies that the desired $*$-homomorphism $\pi:\ep \rightarrow \q$ exists.

Now we prove that $\pi$ is an isomorphism. In order to do this, it suffices to find a homomorphism $\rho :\q\rightarrow \ep$ such that $\rho \circ \pi=\mathrm{id}_{\ep}$ and $\pi \circ \rho=\mathrm{id}_{\q}$. For any $(\mu,g)\in \lag$ and $X=(\cup_{i=1}^l \mu_i\La) \times G\in \mj(\lag)$, we define
$$T_{(\mu,g)}:=s_\mu u_g  \hspace{5mm} \mathrm{and} \hspace{5mm} Q_X:=\sum_{i=1}^l s_{\mu_i}s_{\mu_i}^*.$$
We will show that the family $\{T,Q\}$ satisfies the properties of Definitions \ref{defn3.1}. Relations (1)-(3) easily hold by the $(G,\La)$-relations for $\{s,u\}$. Also, for every $X=(\cup_{i=1}^l \mu_i\La) \times G$ and $Y=(\cup_{j=1}^{l'}\nu_j\La) \times G$ in $\mj(\lag)$, we have
\begin{align*}
Q_X Q_Y&=\big(\sum_{i=1}^l s_{\mu_i}s_{\mu_i}^* \big)\big(\sum_{j=1}^{l'} s_{\nu_j}s_{\nu_j}^* \big)\\
&=\sum_{i,j} s_{\mu_i}(s_{\mu_i}^* s_{\nu_j})s_{\nu_j}^*\\
&=\sum_{i,j} s_{\mu_i}\big(\sum_{(\alpha,\beta)\in \La^{\mathrm{min}}(\mu_i,\nu_j)}s_\alpha s_\beta^* \big) s_{\nu_j}^* \hspace{5mm} (\mathrm{by ~ [1, Lemma~ 3.3]})\\
&=\sum_{i,j} \sum_{\substack{\mu_i\alpha=\nu_j \beta\\ d(\mu_i\alpha)=d(\mu_i)\vee d(\nu_j)}} s_{\mu_i\alpha} s_{\nu_j \beta}^*\\
&=Q_{X\cap Y} \hspace{15mm} (\mathrm{by ~ Lemma} ~ \ref{lem3.5}).
\end{align*}

For eq. (5) of Definition \ref{defn3.1}, let $F=\{X_i\times G:=\cup_{j=1}^{t_i}\mu_{ij}\La \times G: 1\leq i\leq l\}$ be a foundation set in $\mj(\lag)$. Then $F'=\{X_i=\cup_{j=1}^{t_i}\mu_{ij}\La\}_{i=1}^l$ is a foundation set in $\mj(\La)$ by Lemma \ref{lem3.4}(2). Defining $n:=\bigvee_{i,j}d(\mu_{ij})$, we claim that the set $M=\{\mu_{ij}\alpha: \alpha\in \La^{n-d(\mu_{ij})}, 1\leq i\leq l, 1\leq j\leq t_i\}$ coincides with $\La^n$. Indeed, if on the contrary there exists some $\la\in \La^n \setminus M$, then $\La^{\mathrm{min}}(\la,\mu_{ij})=\emptyset$, and hence $\la\La \cap \mu_{ij}\La=\emptyset$ for all $i$ and $j$. This yields that $\la\La\cap X_i=\emptyset$ for every $X_i\in F'$, contradicting that $F'$ is a foundation set in $\mj(\La)$.

Now one may compute
\begin{align*}
\prod_{X_i\times G\in F} (1-Q_{X_i\times G})&= \prod_{i=1}^l (1-\sum_{j=1}^{t_i} s_{\mu_{ij}} s_{\mu_{ij}}^*)\\
&= \prod_{i=1}^l \big(1- \sum_{j=1}^{t_i} s_{\mu_{ij}} (\sum_{\alpha\in \La^{n-d(\mu_{ij})}} s_\alpha s_\alpha^*) s_{\mu_{ij}}^*  \big)\\
&=\prod_{i=1}^l \big(1-\sum_{j=1}^{t_i} \sum_{\alpha\in \La^{n-d(\mu_{ij})}} s_{\mu_{ij}\alpha} s_{\mu_{ij}\alpha}^*  \big) \hspace{10mm} (\star).
\end{align*}
Observe that the projections $s_{\mu_{ij}\alpha}s_{\mu_{ij}\alpha}^*$ are pairwise orthogonal because $d(\mu_{ij}\alpha)=n$ for all $i,j$ (see \cite[Remark 3.2(c)]{pin13}). Hence, using the above claim, expression $(\star)$ equals to
$$(\star)=1-\sum_{\la\in\La^n}s_\la s_\la^*=0 \hspace{10mm} (\mathrm{by ~ (KP4)}).$$
Therefore, the family $\{T,Q\}$ satisfies the relations of Definition \ref{defn3.1}, and by the universality there exists an algebra $*$-homomorphism $\rho:\q\rightarrow \ep$ such that $\rho(t_{(\mu,g)})=T_{(\mu,g)}$ and $\rho(q_X)=Q_X$ for $(\mu,g)\in \lag$ and $X\in \mj(\lag)$. It is clear that $\rho\circ \pi=\mathrm{id}_{\ep}$ and $\pi\circ \rho=\mathrm{id}_{\q}$ because they fix the generators of $\ep$ and $\q$, respectively. Consequently, $\pi$ is an isomorphism, completing the proof.
\end{proof}


\section{A graded Uniqueness Theorem}

In this section, we prove a graded uniqueness theorem for $\ep$ which generalizes and modifies \cite[Theorem A]{haz19} for self-similar $k$-graphs. This modification, in particular, helps us to prove Theorems \ref{steib-algeb-thm} and \ref{thm6.7}.

Let us first recall some definitions. Let $\Gamma$ be a group and $A$ be an algebra over a ring $R$. $A$ is called \emph{$\Gamma$-graded} (or briefly, \emph{graded} whenever the group is clear) if there is a family of $R$-submodules $\{A_\gamma:\gamma\in \Gamma\}$ of $A$ such that $A=\bigoplus_{\gamma\in \Gamma} A_\gamma$ and $A_\gamma A_{\gamma'}\subseteq A_{\gamma\gamma'}$ for all $\gamma,\gamma'\in \Gamma$. Then each set $A_\gamma$ is called a \emph{$\gamma$-homogeneous component of $A$}. In this case, we say an ideal $I$ of $A$ is \emph{$\Gamma$-graded} if $I=\bigoplus_{\gamma\in \Gamma}(I\cap A_\gamma)$. Note that an ideal $I$ of $A$ is $\Gamma$-graded if and only if it is generated by a subset of $\bigcup_{\gamma\in \Gamma} A_\gamma$, the homogeneous elements of $A$.

Furthermore, if $A$ and $B$ are two $\Gamma$-graded algebras over $R$, a homomorphism $\phi:A\rightarrow B$ is said to be a \emph{graded homomorphism} if $\phi(A_\gamma)\subseteq B_\gamma$ for all $\gamma\in \Gamma$. Hence the kernel of a graded homomorphism is always a graded ideal. Also, if $I$ is a graded ideal of $A$, then there is a natural $\Gamma$-grading $(A_\gamma+I)_{\gamma\in \Gamma}$ on the quotient algebra $A/I$, and thus the quotient map $A\rightarrow A/I$ is a graded homomorphism.

\begin{lem}
Let $(G,\La)$ be a self-similar $k$-graph. If for every $n\in \Z^k$, we define
$$\ep_n:=\mathrm{span}_R\big\{s_\mu u_{s(\mu),g} s_\nu^*: g\in G,~ \mu,\nu\in \La,~ \mathrm{and}~ d(\mu)-d(\nu)=n\big\},$$
then $(\ep_n)_{n\in \Z^k}$ is a $\Z^k$-grading on $\ep$.
\end{lem}

\begin{proof}
Consider the free $*$-algebra $\mathbb{F}_R(X)$  and its ideal $I$ as in Subsection \ref{ss2.2}. If we define
$$\theta(S_\mu):=d(\mu),~ \theta(S_\mu^*):=-d(\mu), ~\mathrm{and} ~ \theta(U_{v,g}):=0$$
for all $g\in G$, $v\in \La^0$ and $\mu\in \La$, then $\theta$ induces a $\Z^k$-grading on $\mathbb{F}_R(X)$. Also, since the generators of $I$ are all homogenous, $I$ is a graded ideal. Therefore, $\ep\cong \mathbb{F}_R(X)/I$ is a $\Z^k$-graded algebra, and Proposition \ref{prop2.7} concludes the result.
\end{proof}

\begin{thm}[Graded Uniqueness]\label{grad-uniq-thm}
Let $(G,\La)$ be a self-similar $k$-graph. Let $\phi:\ep\rightarrow B$ be a $\Z^k$-graded $R$-algebra $*$-homomorphism into a $\Z^k$-graded $*$-algebra $B$. If $\phi(a)\neq 0$ for every nonzero element of the form $a=\sum_{i=1}^l r_i u_{v,g_i}$ with $v\in\La^0$ and $g_i^{-1} \cdot v =g_j^{-1} \cdot v $ for $1\leq i,j\leq l$, then $\phi$ is injective.
\end{thm}

\begin{proof}
For convenience we write $A=\ep$. Since $A=\bigoplus_{n\in \Z^k}A_n$ and $\phi$ preserves the grading, it suffices to show that $\phi$ is injective on each $A_n$. So, fix some $b\in A_n$, and assume $\phi(b)=0$. By equation (\ref{eq2.1}), we can write
\begin{equation}\label{equ3.1}
b=\sum_{i=1}^lr_i s_{\mu_i}u_{w_i,g_i}s_{\nu_i}^*
\end{equation}
where $w_i=s(\mu_i)=g_i \cdot s(\nu_i)$ and $d(\mu_i)-d(\nu_i)=n$ for $1\leq i\leq l$. Define $n'=\vee_{1\leq i\leq l}d(\mu_i)$. Then, for each $i\in \{1\ldots l\}$, (KP4) says that
$$s_{w_i}=\sum_{\la\in w_i\La^{n'-d(\mu_i)}}s_\la s_\la^*,$$
and we can write
\begin{align*}
s_{\mu_i}u_{w_i,g_i}s_{\nu_i}^*&=s_{\mu_i}(s_{s(\mu_i)})u_{w_i,g_i}s_{\nu_i}^*\\
&=\sum s_{\mu_i}(s_\la s_\la^*)u_{w_i,g_i}s_{\nu_i}^*\\
&=\sum s_{\mu_i \la}\big(s_{\nu_i} u_{w_i,g_i}^*s_{\la}\big)^*\\
&=\sum s_{\mu_i \la}\big(s_{\nu_i}u_{g_i^{-1} \cdot w_i,g_i^{-1}} s_\la\big)^*\\
&=\sum s_{\mu_i \la}\big(s_{\nu_i} s_{g_i^{-1} \cdot \la}u_{g_i^{-1} \cdot s(\la),\varphi(g_i^{-1},\la)} \big)^*\\
&=\sum s_{\mu_i\la} u_{\varphi(g_i^{-1},\la)^{-1}g_i^{-1} \cdot s(\la),\varphi(g_i^{-1},\la)^{-1}} s_{\nu_i (g_i^{-1} \cdot \la)}^*
\end{align*}
where the above summations are on $\la\in w_i\La^{n'-d(\mu_i)}$. So, in each term of (\ref{equ3.1}), we may assume $d(\mu_i)=n'$ and $d(\nu_i)=n'-n$. Now, for any $1\leq j\leq l$, (KP3) yields that
$$s_{\mu_j}^*b s_{\nu_j}=s_{\mu_j}^*\big( \sum_{i=1}^lr_i s_{\mu_i}u_{w_i,g_i}s_{\nu_i}^*\big) s_{\nu_j}= \sum_{i\in [j]}r_i u_{w_i,g_i}, $$
where $[j]:=\{1\leq i\leq l:(\mu_i,\nu_i)=(\mu_j,\nu_j)\}$. Thus
$$\phi\big(\sum_{i\in [j]}r_i u_{w_i,g_i} \big)=\phi(s_{\mu_j}^*)\phi(b)\phi(s_{\nu_j})=0$$
and hypothesis forces $\sum_{i\in [j]}r_i u_{w_i,g_i}=0$. Therefore,
$$\sum_{i\in [j]}r_i s_{\mu_i}u_{w_i,g_i}s_{\nu_i}^*=s_{\mu_j}\big(\sum_{i\in [j]}r_i u_{w_i,g_i} \big) s_{\nu_j}^*=0.$$
Since the index set $\{1,\ldots,l\}$ is a disjoint union of $[j]$'s, we obtain $b=0$. It follows that $\phi$ is injective.
\end{proof}


\section{$\ep$ as a Steinberg algebra}\label{sect5}

In this section, we want to prove an Steinberg algebra model for $\ep$. Although our result will be the $k$-graph generalization of \cite[Theorem B]{haz19}, note that our proof relies on the graded uniqueness theorem, Theorem \ref{grad-uniq-thm}, and is completely deferent from that of \cite[Theorem B]{haz19}. This gives us a much easier and shorter proof.

Let us first review some terminology about groupoids; see \cite{ren80} for more details. A \emph{groupoid} is a small category $\mathcal{G}$ with inverses. For each $\alpha\in \mathcal{G}$, we may define the range $r(\alpha):=\alpha \alpha^{-1}$ and the source $s(\alpha):=\alpha^{-1}\alpha$ satisfying $r(\alpha)\alpha=\alpha=\alpha s(\alpha)$. It follows that for every $\alpha,\beta\in \mathcal{G}$, the composition $\alpha\beta$ is well-defined if and only if $s(\alpha)=r(\beta)$. The \emph{unit space of $\mathcal{G}$} is $\mathcal{G}^{(0)}:=\{\alpha^{-1}\alpha: \alpha\in \mathcal{G}\}$. Throughout the paper we work with \emph{topological groupoids}, which are ones equipped with a topology such that the maps $r$ and $s$ are continuous. Then a \emph{bisection} is a subset $B\subseteq \mathcal{G}$ such that both restrictions $r|_B$ and $s|_B$ are homeomorphisms. In case $\mathcal{G}$ has a basis of compact open bisections, $\mathcal{G}$ is called an \emph{ample groupoid}.

Let $(G,\La)$ be a self-similar $k$-graph. We also recall the groupoid $\g$ introduced in \cite{li18}. Let $C(\N^k,G)$ be the group of all maps form $\N^k$ to $G$ with the pointwise multiplication. For $f,g\in C(\N^k,G)$, define the equivalence relation $f\sim g$ in case there exists $n_0\in \N^k$ such that $f(n)=g(n)$ for all $n\geq n_0$. Write $Q(\N^k,G):=C(\N^k,G)/\sim$. Also, for each $z\in\Z^k$, let $\mathcal{T}_z: C(\N^k,G)\rightarrow C(\N^k,G)$ be the automorphism defined by
$$\mathcal{T}_z(f)(n)=\left\{
    \begin{array}{ll}
      f(n-z) & n-z\geq 0 \\
      e_G & \mathrm{otherwise}
    \end{array}
  \right. \hspace{5mm} (f\in C(\N^k,G),~ n\in \N^k).
$$
Then $\mathcal{T}_z$ induces an automorphism, denoted again by $\mathcal{T}_z$, on $Q(\N^k,G)$, which is $\mathcal{T}_z([f])=[\mathcal{T}_z(f)]$. So, $\mathcal{T}:\Z^k\rightarrow \mathrm{Aut}Q(\N^k,G)$ is a homomorphism and we consider the semidirect product group $Q(\N^k,G)\rtimes_{\mathcal{T}} \Z^k$.

Note that for every $g\in G$ and $x\in \La^\infty$, one may define $\varphi(g,x)\in C(\N^k,G)$ by
$$\varphi(g,x)(n):=\varphi(g,x(0,n)) \hspace{5mm}  (n\in \N^k).$$
Moreover, \cite[Lemma 3.7]{li18} says that there exists a unique action $G\curvearrowright \La^\infty$ by defining
$$(g\cdot x)(m,n):=\varphi(g,x(0,m)) \cdot x(m,n)  \hspace{7mm} ((m,n)\in \Omega_k)$$
for every $g\in G$ and $x\in\La^\infty$.


\begin{defn}\label{defn5.1}
A self-similar $k$-graph $(G,\La)$ is said to be \emph{pseudo free} if for any $g\in G$ and $\mu\in \La$, $g\cdot \mu=\mu$ and $\varphi(g,\mu)=e_G$ imply $g=e_G$.
\end{defn}

According to \cite[Lemma 5.6]{li18}, in case $(G,\La)$ is pseudo free, then we have
$$g\cdot \mu=h \cdot \mu ~~ \mathrm{and} ~~ \varphi(g,\mu)=\varphi(h,\mu) ~~ \Longrightarrow ~~ g=h$$
for every $g,h\in G$ and $\mu\in \La$.

\begin{defn}
Associated to $(G,\La)$ we define the subgroupoid
$$\g:=\left\{\big(\mu(g\cdot x);\mathcal{T}_{d(\mu)}([\varphi(g,x)]),d(\mu)-d(\nu);\nu x\big) :g\in G, ~\mu,\nu\in \La,~ s(\mu)=g\cdot s(\nu) \right\}$$
of $\La^\infty\times \big(Q(\N^k,G)\rtimes_{\mathcal{T}} \Z^k\big) \times \La^\infty$ with the range and source maps
$$r(x;[f],n-m;y)=x  \hspace{5mm} \mathrm{and} \hspace{5mm} s(x;[f],n-m;y)=y.$$
\end{defn}

Note that if we set
$$Z(\mu,g,\nu):=\left\{\big(\mu(g\cdot x);\mathcal{T}_{d(\mu)}([\varphi(g,x)]),d(\mu)-d(\nu);\nu x\big) :x\in s(\nu)\La^\infty \right\},$$
then the basis
$$\mathcal{B}_{G,\La}:=\left\{Z(\mu,g,\nu):\mu,\nu\in \La,g\in G,s(\mu)=g\cdot s(\nu)\right\}$$
induces a topology on $\g$. In case $(G,\La)$ is pseudo free, \cite[Proposition 3.11]{li19} shows that $\g$ is a Hausdorff groupoid with compact open base $\mathcal{B}_{G,\La}$.

\begin{defn}
Let $(G,\La)$ be a pseudo free self-similar $k$-graph and $R$ a unital commutative $*$-ring. Then \emph{the Steinberg algebra associated to $(G,\La)$} is the $R$-algebra
$$A_R(\g):=\mathrm{span}_R\{1_B: B ~ \mathrm{is ~ a ~ compact ~ open ~ bisection}\}$$
endowed with the pointwise addition, the multiplication $fg(\gamma):=\sum_{\alpha\beta=\gamma}f(\alpha) g(\beta)$, and the involution $f^*(\gamma):=f(\gamma^{-1})^*$ for all $\gamma\in \g$.
\end{defn}

To prove Theorem \ref{steib-algeb-thm} below, we need the following lemma.

\begin{lem}\label{lem4.2}
Let $(G,\La)$ be a pseudo free self-similar $k$-graph. Let $v,w\in \La^0$ and $g,h\in G$ with $g\cdot v=h\cdot v=w$. Then $Z(v,g,w)\cap Z(v,h,w)= \emptyset $ whenever $g\neq h$.
\end{lem}

\begin{proof}
Suppose that $(g\cdot x;[\varphi(g,x)],0;x)=(h \cdot y,[\varphi(h,y)],0;y)\in Z(v,g,w)\cap Z(v,h,w)$ where $x,y\in Z(w)$. Then $y=x$, $g\cdot x=h \cdot x$ and $[\varphi(g,x)]=[\varphi(h,x)]$. Since $(G,\La)$ is pseudo free, \cite[Corollary 5.6]{li18} implies that $g=h$.
\end{proof}

\begin{thm}\label{steib-algeb-thm}
Let $(G,\La)$ be a pseudo free self-similar $k$-graph. Then there is a (unique) $*$-algebra isomorphism $\phi:\ep\rightarrow A_R(\g)$ such that
$$\phi(s_\mu)=1_{Z(\mu, e_G, s(\mu))} \hspace{5mm} \mathrm{and} \hspace{5mm}  \phi(u_{v,g})=1_{Z(v,g,g^{-1} \cdot v )}$$
for every $\mu\in\La$, $v\in\La^0$, and $g\in G$. In particular, the elements $rs_\mu$ and $ru_{v,g}$ with $r\in R\setminus\{0\}$ are all nonzero.
\end{thm}

\begin{proof}
For each $v\in\La^0$, $\mu\in\La$ and $g\in G$, define
$$S_\mu:=1_{Z(\mu, e_G, s(\mu))} \hspace{5mm} \mathrm{and} \hspace{5mm} U_{v,g}:=1_{Z(v,g,g^{-1} \cdot v )}.$$
Since $S_\mu^*=1_{Z(\mu,e_G,s(\mu))^{-1}}=1_{Z(s(\mu),e_G,\mu)}$ and $U_{v,g}^*=1_{Z(v,g,g^{-1} \cdot v )^{-1}}=1_{Z(g^{-1} \cdot v ,g^{-1},v)}$, a long but straightforward computation shows that $\{S_\mu,U_{v,g}\}$ is a $(G,\La)$-family in $A_R(\g)$. Then, by the universal property, such $*$-homomorphism $\phi$ exists.

\cite[Proposition 3.11]{li19} says that $\g$ is ample with compact open base $\mathcal{B}_{G,\La}$. Since each element $Z(\mu,g,\nu)$ of $\mathcal{B}_{G,\La}$ can be written as
$$Z(\mu,g,\nu)=Z(\mu,e_G,s(\mu))Z(s(\mu),g,s(\nu))Z(\nu,e_G,s(\nu))^{-1},$$
$\phi$ is surjective.

We will show the injectivity of $\phi$ by applying the graded uniqueness theorem. Note that the continuous 1-cocycle $c:\g\rightarrow \Z^k$, defined by $c(\mu(g\cdot x);[f],d(\mu)-d(\nu);\nu x):=d(\mu)-d(\nu)$, induces a $\Z^k$-grading on $A_R(\g)$. Also, $\phi$ preserves the $\Z^k$-grading because it does on the generators. Now, to apply Theorem \ref{grad-uniq-thm}, we assume $\phi(a)=0$ for an element of the form $a=\sum_{i=1}^l r_i u_{v,g_i}$ with $g_i^{-1} \cdot v =g_j^{-1} \cdot v $ for $1\leq i,j\leq l$. We may also assume that the $g_i$'s are distinct (otherwise, combine the terms with same $g_i$'s). We then have
$$\phi(a)=\sum_{i=1}^l r_i1_{Z(v,g_i,g_i^{-1} \cdot v )}=0.$$
Lemma \ref{lem4.2} says that the bisections $Z(v,g_i,g_i^{-1} \cdot v )$ are pairwise disjoint. Hence, for each $i$, if we pick some $\alpha\in Z(v,g_i,g_i^{-1} \cdot v )$, then $r_i=\phi(a)(\alpha)=0$. Therefore $a=0$, and Theorem \ref{grad-uniq-thm} concludes that $\phi$ is injective. We are done.
\end{proof}

Combining \cite[Theorem 6.7]{ste10}, \cite[Theorem 5.9]{li18}, and Theorem \ref{steib-algeb-thm} gives the next corollary. (Although in \cite{li18} it is supposed $|\La^0|<\infty$, but \cite[Theorem 5.9]{li18} holds also for $\La$ with infinitely many vertices.)

\begin{cor}
Let $(G,\La)$ be a pseudo free self-similar $k$-graph over an amenable group $G$. Then the complex algebra $\mathrm{EP}_\C(G,\La)$ is a dense subalgebra of $\mathcal{O}_{G,\La}$ introduced in \cite{li19}.
\end{cor}

In the following, we see that the Kumjian-Pask algebra $\mathrm{KP}_R(\La)$ from \cite{pin13} can be embedded in $\ep$.

\begin{cor}
Let $(G,\La)$ be a pseudo free self-similar $k$-graph. Let the Kumjian-Pask algebra $\mathrm{KP}_R(\La)$ be generated by a Kumjian-Pask $\La$-family $\{t_\mu:\mu\in\La\}$. Then the map $t_\mu\mapsto s_\mu$ embeds $\mathrm{KP}_R(\La)$ into $\ep$ as a $*$-subalgebra.
\end{cor}

\begin{proof}
We know that $\mathrm{KP}_R(\La)$ is $\Z^k$-graded by the homogenous components
$$\mathrm{KP}_R(\La)_n:=\mathrm{span}_R \left\{t_\mu t_\nu^*:\mu,\nu\in \La, d(\mu)-d(\nu)=n\right\}.$$
for all $n\in\Z^k$. Then, the universal property of Kumjian-Pask algebras gives a graded $*$-algebra homomorphism $\phi:\mathrm{KP}_R(\La)\rightarrow \ep$ such that $\phi(t_\mu):=s_\mu$ and $\phi(t_\mu^*):=s_\mu^*$ for every $\mu\in\La$. Moreover, Theorem \ref{steib-algeb-thm} shows that $\phi(rt_\mu)=rs_\mu\neq 0$ for all $ r\in R\setminus \{0\}$ and $\mu\in \La$. Therefore, the graded uniqueness theorem for Kumjian-Pask algebras \cite[Theorem 4.1]{pin13} implies that $\phi$ is injective.
\end{proof}

\begin{defn}
Let $\mathcal{G}$ be a topological groupoid. We say that $\mathcal{G}$ is \emph{topologically principal} if the set of units with trivial isotropy group, that is $\{u\in \mathcal{G}^{(0)}: s^{-1}(u)\cap r^{-1}(u)=\{u\}\}$, is dense in $\mathcal{G}^{(0)}$.
\end{defn}

The analogue of the topologically principal property for self-similar $k$-graphs is $G$-aperiodicity (see \cite[Proposition 6.5]{li18}).

\begin{defn}
Let $(G,\La)$ be a self-similar $k$-graph. Then $\La$ is said to be \emph{$G$-aperiodic} if for every $v\in \La^0$, there exists $x\in v\La^\infty$ with the property that
$$x(p,\infty)=g\cdot x(q,\infty) ~~\Longrightarrow ~~  g=e_G ~~ \mathrm{and} ~~ p=q \hspace{5mm} (\forall g\in G,~ \forall p,q\in \N^k).$$
\end{defn}

\begin{thm}[The Cuntz-Krieger uniqueness]
Let $(G,\La)$ be a pseudo free self-similar $k$-graph. Let $(G,\La)$ be also $G$-aperiodic. Suppose that $\phi:\ep\rightarrow A$ is a $*$-algebra homomorphism from $\ep$ into a $*$-algebra $A$ such that $\phi(rs_v)\neq 0$ for all $0\neq r\in R$ and $v\in \La^0$. Then $\phi$ is injective.
\end{thm}

\begin{proof}
First note that $\g$ is a Hausdorff ample groupoid by \cite[Proposition 3.11]{li19}, and that $\mathcal{B}_{G,\La}$ is a basis for $\g$ consisting compact open bisections. Also, \cite[Lemma 3.12]{li19} says that $\g$ is topologically principal (so is effective in particular). So, we may apply \cite[Theorem 3.2]{cla15}.

Denote by $\psi:\ep\rightarrow A_R(\g)$ the isomorphism of Theorem \ref{steib-algeb-thm}. If on the contrary $\phi$ is not injective, then neither is $\widetilde{\phi}:=\phi \circ \psi^{-1}:A_R(\g)\rightarrow A$. Thus, by \cite[Theorem 3.2]{cla15}, there exists a compact open subset $K\subseteq \g^{(0)}$ and $r\neq 0$ such that $\widetilde{\phi}(r1_K)=0$. Since $K$ is open, there is a unit $U=Z(\mu,e_G,\mu)\in \mathcal{B}_{G,\La}$ such that $U\subseteq K$. So we get
$$\phi(rs_\mu s_\mu^*)=\widetilde{\phi}(r1_U)=\widetilde{\phi}(r1_{U\cap K})=\widetilde{\phi}(r1_K) \widetilde{\phi}(1_U)=0,$$
and hence
$$\phi(rs_{s(\mu)})=\phi(s_\mu^*)\phi(rs_\mu s_\mu^*)\phi(s_\mu)=0.$$
This contradicts the hypothesis, and therefore $\phi$ is injective.
\end{proof}


\section{Ideal structure}

By an \emph{ideal} we mean a two-sided and self-adjoint one. In this section, we characterize basic, $\Z^k$-graded and diagonal-invariant ideals of $\ep$, which are exactly all basic $Q(\N^k,G)\rtimes_{\mathcal{T}}\Z^k$-graded ones.

Let $(G,\La)$ be a pseudo free self-similar $k$-graph. Since $\g$ is a Hausdorff ample groupoid \cite[Theorem 5.8]{li18}, $\g^{(0)}$ is both open and closed, and for every $f\in A_R(\g)$ the restricted function $f|_{\g^{(0)}}=f\chi_{\g^{(0)}}$ lies again in $A_R(\g)$. Then $A_R(\g^{(0)})$ is a $*$-subalgebra of $A_R(\g)$ and there is a conditional expectation $\e:A_R(\g)\rightarrow A_R(\g^{(0)})$ defined by $\e(f)=f|_{\g^{(0)}}$ for $f\in A_R(\g)$. Let $D:=\mathrm{span}_R\{s_\mu s_\mu^*:\mu\in\La\}$ be the diagonal of $\ep$. In light of Theorem \ref{steib-algeb-thm}, it is easy to check that the expectation is $\e:\ep\rightarrow D$ defined by
$$\e\big(s_\mu u_{s(\mu),g}s_\nu^*\big)=\delta_{\mu,\nu}\delta_{g,e_G} s_\mu s_\mu^* \hspace{5mm}  (\mu,\nu\in \La,~s(\mu)=g\cdot s(\nu)).$$

\begin{defn}
An ideal $I$ of $\ep$ is called {\it diagonal-invariant} whenever $\e(I)\subseteq I$. Also, $I$ is said to be \emph{basic} if $rs_v\in I$ implies $s_v\in I$ for all $v\in \La^0$ and $r\in R\setminus \{0\}$.
\end{defn}

\begin{defn}
Let $(G,\La)$ be a self-similar $k$-graph. A subset $H\subseteq \La^0$ is called
\begin{enumerate}
  \item \emph{$G$-hereditary} if $r(\mu)\in H ~~ \Longrightarrow ~~ g\cdot s(\mu)\in H$ for all $g\in G$ and $\mu\in \La$;
  \item \emph{$G$-saturated} if $v\in \La^0$ and $s(v\La^n) \subseteq H~~ \mathrm{for ~ some ~} n\in\N^k ~~ \Longrightarrow ~~ v\in H$.
\end{enumerate}
\end{defn}

In the following, given any $H\subseteq \La^0$, we denote by $I_H$ the ideal of $\ep$ generated by $\{s_v:v\in H\}$. Also, for each ideal $I$ of $\ep$, we define $H_I:=\{v\in \La^0: s_v\in I\}$.

To prove Theorem \ref{thm6.7} we need some structural lemmas about the ideals $I_H$ and associated quotients $\ep/I_H$.

 \begin{lem}\label{lem5.2}
If $I$ is an ideal of $\ep$, then $H_I:=\{v\in \Lambda^0:s_v\in I\}$ is a $G$-saturated $G$-hereditary subset of $\La^0$.
 \end{lem}

\begin{proof}
The proof is straightforward.
\end{proof}

\begin{lem}\label{lem5.3}
Let $H$ be a $G$-saturated $G$-hereditary subset of $\La^0$ and $I_H$ the ideal of $\ep$ generated by $\{s_v:v\in H\}$. Then we have
\begin{equation}\label{equ5.1}
I_H=\mathrm{span}_R\left\{s_\mu u_{s(\mu),g}s_\nu^*:g\in G, ~ s(\mu)=g\cdot s(\nu)\in H\right\},
\end{equation}
and $I_H$ is a $\Z^k$-graded diagonal-invariant ideal.
\end{lem}

\begin{proof}
Denote by $J$ the right-hand side of (\ref{equ5.1}). The identity
$$s_\mu u_{s(\mu),g}s_\nu^*=s_\mu(s_{s(\mu)})u_{s(\mu),g}s_\nu^*$$
yields $J\subseteq I_H$. Also, using the description of $\ep$ in Proposition \ref{prop2.7}, it is straightforward to check that $J$ is an ideal of $\ep$. So, by $s_v=s_vu_{v,e_G}s_v^*$, $J$ contains all generators of $I_H$, and we have proved (\ref{equ5.1}).

Now, (\ref{equ5.1}) says that $I_H$ is spanned by its homogenous elements, hence it is a graded ideal. Moreover, let $a=\sum_{i=1}^l s_{\mu_i}u_{s(\mu_i),g_i}s_{\nu_i}^*\in I_H$ such that $s(\mu_i)=g\cdot s(\nu_i)\in H$. Then, in particular, each term of $a$ with $g_i=e_G$ belongs to $I_H$. Therefore, $\mathcal{E}(a)\in I_H$, and $I_H$ is diagonal-invariant.
\end{proof}

Let $H$ be a $G$-saturated $G$-hereditary subset of $\La^0$ and consider the $k$-subgraph $\La\setminus \La H$. Then the restricted action $G\curvearrowright \La\setminus \La H$ is well-defined, and hence $(G,\La\setminus \La H, \varphi|_{G\times \La\setminus \La H})$ is also a self-similar $k$-graph. So we have:

\begin{lem}\label{lem5.4}
Let $(G,\La)$ be a pseudo free self-similar $k$-graph. If $H$ is a $G$-saturated $G$-hereditary subset of $\La^0$, then $(G,\La\setminus\La H)$ is a pseudo free self-similar $k$-graph.
\end{lem}

\begin{proof}
The proof is straightforward.
\end{proof}

\begin{lem}\label{lem5.5}
Let $H$ be a $G$-saturated $G$-hereditary subset of $\La^0$. For every $v\in \La^0$ and $r\in R\setminus\{0\}$, $rs_v\in I_H$ implies $v\in H$.
\end{lem}

\begin{proof}
Let $\{t_\mu, w_{v,g}\}$ be the generators of $\mathrm{EP}_R(G,\La\setminus\La H)$. If we define
$$S_\mu:=\left\{
    \begin{array}{ll}
      t_\mu & s(\mu)\notin H \\
      0 & \mathrm{otherwise}
    \end{array}
  \right.
\hspace{5mm} \mathrm{and} \hspace{5mm}
U_{v,g}:=\left\{
    \begin{array}{ll}
      w_{v,g} & v\notin H \\
      0 & \mathrm{otherwise,}
    \end{array}
  \right.
$$
then $\{S_\mu,U_{v,g}\}$ is a $(G,\La)$-family in $\mathrm{EP}_R(G,\La\setminus\La H)$, and by the universality, there is a $*$-homomorphism $\psi:\ep\rightarrow \mathrm{EP}_R(G,\La\setminus\La H)$ such that $\psi(s_\mu)=S_\mu$ and $\psi(u_{v,g})=U_{v,g}$ for all $\mu\in\La$, $v\in\La^0$ and $g\in G$. Since $\psi(s_v)=0$ for every $v\in H$, we have $I_H\subseteq \ker \psi$. On the other hand, Theorem \ref{steib-algeb-thm} implies that all $\psi(rs_v)=rt_v$ are nonzero for $v\in \La^0\setminus H$ and $r\in R\setminus\{0\}$.

Now assume $rs_v\in I_H$ for some $v\in \La^0$ and $r\in R\setminus\{0\}$. If $v\in \La^0\setminus H$, then $\psi(rs_v)=rt_v\neq 0$, and we get $rs_v\notin \ker \psi \supseteq I_H$, a contradiction.
\end{proof}

In fact, Lemma \ref{lem5.5} says that $I_H$ is a basic ideal with $H_{I_H}=H$ for every $G$-saturated $G$-hereditary subset $H$ of $\La^0$.

\begin{prop}\label{prop5.6}
Let $H$ be a $G$-saturated $G$-hereditary subset of $\La^0$. Let $\{t_\mu,w_{v,g}\}$ be the $(G,\La\setminus\La H)$-family generating $\mathrm{EP}_R(G,\La\setminus \La H)$. Then the map $\psi: \mathrm{EP}_R(G,\La\setminus \La H) \rightarrow \ep/I_H$ defined by
$$\psi(t_\mu w_{s(\mu),g} t_\nu^*):=s_\mu u_{s(\mu),g} s_\nu^* +I_H \hspace{5mm} (\mu,\nu\in \La\setminus \La H,~ g\in G)$$
is an ($R$-algebra) $*$-isomorphism.
\end{prop}

\begin{proof}
If we set $T_\mu:= s_\mu+I_H$ and $W_{v,g}:=u_{v,g}+I_H$ for every $v\in \La^0,~\mu\in\La$, and $g\in G$, then $\{T_\mu,W_{v,g}\}$ is a $(G,\La\setminus\La H)$-family in $\ep/I_H$ (the relations of Definition \ref{defn2.4} for $\{T_\mu, W_{v,g}\}$ immediately follow from those for $\{s_\mu,u_{v,g}\}$). So, the universality of $\mathrm{EP}_R(G,\La\setminus \La H)$ gives such $*$-homomorphism $\psi$. Note that $s_\mu\in I_H$ for each $\mu\in \La H$ by (\ref{equ5.1}), which gives the surjectivity of $\psi$.

To prove the injectivity, we apply the graded uniqueness theorem, Theorem \ref{grad-uniq-thm}. First, since $I_H$ is a $\Z^k$-graded ideal, $\ep/I_H$ has a natural $\Z^k$-grading and $\psi$ is a graded homomorphism. Thus, we fix an element in $\mathrm{EP}_R(G,\La\setminus \La H)$ of the form $a=\sum_{i=1}^l r_iw_{v,g_i}$ such that $v\in \La^0\setminus H$ and $g_i^{-1} \cdot v =g_j^{-1} \cdot v $ for all $1\leq i,j\leq l$. Without loss of generality, we may also suppose that the $g_i$'s are distinct. If $\psi(a)=0$, then $\psi(a)=\sum_{i=1}^l r_iu_{v,g_i}+I_H=I_H$ and $\sum_{i=1}^l r_iu_{v,g_i}\in I_H$. Thus, for each $1\leq j\leq l$, we have
$$\big(\sum_{i=1}^l r_iu_{v,g_i}\big)u_{g_j^{-1} \cdot v , g_j^{-1}}=\sum_{i=1}^l r_iu_{v,g_ig_j^{-1}}\in I_H  \hspace{5mm} (\mathrm{by ~ eq.~ (5) ~ in~ Definition~ \ref{defn2.4}})$$
and since $I_H$ is diagonal-invariant,
$$r_js_v=r_ju_{v,e_G}=\e(\sum_{i=1}^l r_iu_{v,g_ig_j^{-1}})\in I_H.$$
As $v\notin H$, Lemma \ref{lem5.5} forces $r_j=0$ for each $1\leq j\leq l$, hence $a=0$. Now Theorem \ref{grad-uniq-thm} implies that $\psi$ is an isomorphism.
\end{proof}

\begin{thm}\label{thm6.7}
Let $(G,\La)$ be pseudo free self-similar $k$-graph. Then $H\mapsto I_H$ is a one-to-one correspondence between $G$-saturated $G$-hereditary subsets of $\La^0$ and basic, $\Z^k$-graded and diagonal-invariant ideals of $\ep$, with inverse $I\mapsto H_I$.
\end{thm}

\begin{proof}
The injectivity of $H\mapsto I_H$ follows from Lemma \ref{lem5.5}. Indeed, if $I_H=I_K$ for $G$-saturated $G$-hereditary subsets $H,K\subseteq \La^0$, then Lemma \ref{lem5.5} yields that $H=H_{I_H}=H_{I_K}=K$.

To see the surjectivity, we take a basic, $\Z^k$-graded and diagonal-invariant ideal $I$ of $\ep$, and then prove $I=I_{H_I}$. Write $J:=I_{H_I}$ for convenience. By Proposition \ref{prop5.6} we may consider $\mathrm{EP}_R(G,\La\setminus\La H_I)\cong \ep/J$ as a $*$-$R$-algebra. Let $\{s_\mu,u_{v,g}\}$ and $\{t_\mu,w_{v,g}\}$ be the generators of $\ep$ and $\ep/J$, respectively. Since $J\subseteq I$, we may define the quotient map $q:\ep/J\rightarrow \ep/I$ such that
$$q(t_\mu)=s_\mu+I \hspace{5mm} \mathrm{and} \hspace{5mm} q(w_{v,g})=u_{v,g}+I$$
for all $\mu\in \La$, $v\in \La^0$ and $g\in G$. Notice that $q$ preserves the grading because $I$ is a $\Z^k$-graded ideal. So, we can apply Theorem \ref{grad-uniq-thm} to show that $q$ is an isomorphism. To do this, fix an element of the form $a=\sum_{i=1}^l r_iw_{v,g_i}$ with $v\in \La^0\setminus H_I$ such that $g_i^{-1} \cdot v =g_j^{-1} \cdot v $ for $1\leq i,j\leq l$ and $q(a)=0$. Then $\sum_{i=1}^l r_iu_{v,g_i}\in I$. As before, we may also assume that the $g_i$'s are distinct. Thus, for each $1\leq j\leq l$, we have
$$b_j:=\big(\sum_{i=1}^l r_iu_{v,g_i}\big)\big(u_{g_j^{-1} \cdot v ,g_j^{-1}}\big)=\sum_{i=1}^l r_iu_{v,g_ig_j^{-1}}\in I.$$
Since $I$ is diagonal-invariant and basic, the case $r_j\neq 0$ yields $r_js_v=r_ju_{v,e_G}=\mathcal{E}(b_j)\in I$, and thus $s_v\in I$ and $v\in H_I$, which contradicts the choice of $v$. It follows that $r_j$'s are all zero, and hence $a=0$. Now Theorem \ref{grad-uniq-thm} implies that $q$ is injective, or equivalently $I=J=I_{H_I}$ as desired.
\end{proof}

In the end, we remark the following about $Q(\N^k,G)\rtimes_{\mathcal{T}}\Z^k$-graded ideals of $\ep$.

\begin{rem}
Let $(G,\La)$ be a pseudo free self-similar $k$-graph and $\g$ be the associated groupoid. Denote by $\Gamma:=Q(\N^k,G)\rtimes_{\mathcal{T}}\Z^k$ the group introduced in Section \ref{sect5}. If we define $c:\g\rightarrow \Gamma$ by $c(x;\gamma;y):=\gamma$, then $c$ is a cocycle on $\g$ because $c(\alpha\beta)=c(\alpha)*_\Gamma c(\beta)$ for all $\alpha,\beta\in \g$ with $s(\alpha)=r(\beta)$. Hence, it induces a $\Gamma$-grading on $A_R(\g)=\ep$ with the homogenous components
$$A_\gamma:=\mathrm{span}_R\{1_V:V\subseteq c^{-1}(\gamma) \mathrm{~is~a~compact~open~bisection~}\}$$
(see \cite[Proposition 5.1]{cla18} for example). By a similar argument as in \cite[\S 6.5]{cla18} and combining Theorem \ref{steib-algeb-thm} and \cite[Theorem 5.3]{cla18}, we may obtain that the ideals of the form $I_H$, described in Theorem \ref{thm6.7} above, are precisely the basic, $\Gamma$-graded ideals of $\ep$.
\end{rem}


{\bf Acknowledgement.} The author is very grateful to the referees for the careful reading and valuable comments. This work was partially supported by Shahid Chamran University of Ahvaz under grant number SCU.MM1400.279.

\section*{Declarations}

{\bf Conflict of interest.} The author declares that he has no conflict of interest.


\end{document}